\documentclass{article}
\usepackage{amsmath,amsthm,amssymb,mathrsfs,amstext}
\theoremstyle{plain}
\newtheorem{thm}{Theorem}
\newtheorem{lem}[thm]{Lemma}
\theoremstyle{definition}
\newtheorem{defn}[thm]{Definition}
\newtheorem{fact}[thm]{Fact}
\newtheorem{prop}[thm]{Proposition}
\providecommand{\keywords}[1]{\textbf{\textit{Keywords:}} #1}
\providecommand{\subjectclass}[2]{\textbf{Mathematics subject classification:} #1}

\title{Nonstandard proof of combined sum and product structure
in IP$^{*}$ sets \footnote{\keywords{IP$^{*}$ sets, Nonstandard analysis}}}

\date{}
\author{Sayan Goswami
\footnote{Department of Mathematics, 
           University of Kalyani, 
           Kalyani-741235,
           Nadia, West Bengal, India
           {\tt sayan92m@gmail.com}}
}

\begin{document}
\maketitle

\begin{abstract}
\noindent V. Bergelson and N. Hindman proved that $IP^{*}$ sets
contain all possible finite sum and product of a sum subsystem of
any sequence in $\mathbb{N}$. In this article, we will prove this
result using Nonstandard analysis.

\end{abstract}

\noindent \subjectclass{05D10, 26E35}

\section{Introduction}
\noindent Ramsey theoretic study is deeply involved with the study of topological
dynamics, ergodic theory and algebra of Stone-\v{C}ech compactification
of discrete semigroup. $IP$ sets play a major role in Ramsey theory.
A subset of $\mathbb{N}$ is said to be an $IP$ set if it contains
a set of the form $FS\langle x_{n}\rangle_{n=1}^{\infty}$, the all
possible finite sum of some sequence $\langle x_{n}\rangle_{n=1}^{\infty}$.
In \cite{key-2}, it was proved that for any finite partition of $\mathbb{N}$,
there is one cell which is an $IP$ set. Similarly $FP\langle x_{n}\rangle_{n=1}^{\infty}$
is defined for products. An $IP^{*}$ set is a set which intersects
all possible $IP$ set. Given a sequence $\langle x_{n}\rangle_{n=1}^{\infty}$
in $\mathbb{N}$, we say that $\langle y_{n}\rangle_{n=1}^{\infty}$
is a sum subsystem of $\langle x_{n}\rangle_{n=1}^{\infty}$ provided
there exists a sequence $\langle H_{n}\rangle_{n=1}^{\infty}$ of
non-empty finite subset such that $\text{max}H_{n}<\text{min}H_{n+1}$
and $y_{n}=\sum_{t\in H_{n}}x_{t}$ for each $n\in\mathbb{N}$. In
\cite{key-1}, the following result has been proved.

\begin{thm}
{\label{Main} }Let $\langle x_{n}\rangle_{n=1}^{\infty}$
be a sequence in $\mathbb{N}$ and $A$ be $IP^{*}$ set in $\left(\mathbb{N},+\right)$.
Then there exists a subsystem $\langle y_{n}\rangle_{n=1}^{\infty}$
of $\langle x_{n}\rangle_{n=1}^{\infty}$ such that 
\[
FS\left(\langle y_{n}\rangle_{n=1}^{\infty}\right)\cup FP\left(\langle y_{n}\rangle_{n=1}^{\infty}\right)\subseteq A.
\]
\end{thm}

In this article we will prove this theorem using the methods of Nonstandard
analysis. Here we recall some basics of this topic. For details readers
are invited to see the recent book \cite{key-1-1}. Nonstandard analysis
is the study of hypernatural numbers, which correlates the ultrafilters
over $\mathbb{N}$. Here we work inside a \textquotedblleft universe\textquotedblright{}
$\mathbb{V}$, equipped with a star map $*:\mathbb{V}\rightarrow\mathbb{V}$
which assigns every object $M\in\mathbb{V}$ to its nonstandard extension
$^{*}M$, in such a way that the following properties hold:
\begin{prop}
For any semigroup $\left(S,\cdot\right)$ the following properties
of the $*$ map can be found in \cite{key-1-1}.
\begin{enumerate}
\item For a set $A$, $^{*}A$, is also a set, and $^{\sigma}\!A\subseteq^{*}\!\!A$,
where $^{\sigma}\!A=\left\{ ^{*}x:x\in A\right\} $. The inclusion
is strict iff $A$ is infinite;
\item If $A,B$ are sets such that $A\subseteq B$, then $^{*}\!A\subseteq\,^{*}\!B$;
\item For a set $A$ and $k\in\mathbb{N}$ we have $^{*}\!\left(A^{k}\right)\subseteq\left(\,^{*}\!A\right)^{k}$;
\item If $f:A\rightarrow B$ is a function, $^{*}f:{}^{*}\!A\rightarrow^{*}\!B$
is a function;
\item If $f:A\rightarrow B$ is a function and $x\in A,$ then $^{*}f\left(x\right)=f(x)$;
\item If $f:A\rightarrow B$ is a function and $x\in A,$ then $^{*}f\left(^{*}x\right)=^{*}\left(f(x)\right)$;
\item If $s\in S$, then $^{*}s=s$;
\item Transfer principle: For any elementary formulae $\varphi\left(x_{1},x_{2},\ldots,x_{m}\right)$
and objects $M_{1},M_{2},\ldots,M_{n}\in\mathbb{V},$ 
\[
\varphi\left(M_{1},M_{2},\ldots,M_{n}\right)\text{ holds iff }\varphi\left(^{*}M_{1},^{*}M_{2},\ldots,^{*}M_{n}\right)\text{ holds}
\]
\end{enumerate}
\end{prop}

Let us first give a brief review of algebraic structure of the Stone-\v{C}ech
compactification of any discrete semigroup $S$.

The set $\{\overline{A}:A\subset S\}$ is a basis for the closed sets
of $\beta S$. The operation `$\cdot$' on $S$ can be extended to
the Stone-\v{C}ech compactification $\beta S$ of $S$ so that$(\beta S,\cdot)$
is a compact right topological semigroup (meaning that for any \ 
is continuous) with $S$ contained in its topological center (meaning
that for any $x\in S$, the function $\lambda_{x}:\beta S\rightarrow\beta S$
defined by $\lambda_{x}(q)=x\cdot q$ is continuous). This is a famous
Theorem due to Ellis that if $S$ is a compact right topological semigroup
then the set of idempotents $E\left(S\right)\neq\emptyset$. A nonempty
subset $I$ of a semigroup $T$ is called a $\textit{left ideal}$
of $S$ if $TI\subset I$, a $\textit{right ideal}$ if $IT\subset I$,
and a $\textit{two sided ideal}$ (or simply an $\textit{ideal}$)
if it is both a left and right ideal. A $\textit{minimal left ideal}$
is the left ideal that does not contain any proper left ideal. Similarly,
we can define $\textit{minimal right ideal}$ and $\textit{smallest ideal}$.

Any compact Hausdorff right topological semigroup $T$ has the smallest
two sided ideal

\[
\begin{array}{ccc}
K(T) & = & \bigcup\{L:L\text{ is a minimal left ideal of }T\}\\
 & = & \,\,\,\,\,\bigcup\{R:R\text{ is a minimal right ideal of }T\}.
\end{array}
\]

Given a minimal left ideal $L$ and a minimal right ideal $R$, $L\cap R$
is a group, and in particular contains an idempotent. If $p$ and
$q$ are idempotents in $T$ we write $p\leq q$ if and only if $pq=qp=p$.
An idempotent is minimal with respect to this relation if and only
if it is a member of the smallest ideal $K(T)$ of $T$. Given $p,q\in\beta S$
and $A\subseteq S$, $A\in p\cdot q$ if and only if the set $\{x\in S:x^{-1}A\in q\}\in p$,
where $x^{-1}A=\{y\in S:x\cdot y\in A\}$. See \cite{key-3} for an
elementary introduction to the algebra of $\beta S$ and for any unfamiliar
details.

Now we correlate the ultrafilters and nonstandard objects.
\begin{defn}
Elements $\alpha\in^{*}\!S$ generate ultrafilters on S via the ultrafilter
map $\alpha\rightarrow\mathcal{U_{\alpha}}$ where $\mathcal{U}_{\alpha}=\left\{ A\subseteq S:\alpha\in^{*}\!A\right\} $.
\end{defn}

Two elements $\alpha,\beta$ are called $u$-equivalent iff $\mathcal{U}_{\alpha}=\mathcal{U}_{\beta}$.
\begin{prop}
$\sim$ is an equivalence relation $^{*}\!S$.
\end{prop}

The relation $\sim$ is an equivalence relation on $^{*}\!S$ was
first considered by Di Nasso in \cite{key-1-0-1}. For details see
\cite{key-1-1}.
\begin{defn}
For any semigroup $S$, equip $^{*}S$ with the $u$-topology, that
generated by the basic open sets $^{*}A$ for $A\subseteq S$. We
say $^{*}S$ is a compact $u$-semigroup, i.e.
\begin{enumerate}
\item $^{*}S$ is compact;
\item For any $\alpha,\beta\in^{*}S$ there is $\gamma\in^{*}S$ such that
$\gamma\sim\alpha+^{*}\beta$;
\item The map $\alpha\rightarrow\alpha+^{*}\beta$ is continuous.
\end{enumerate}
\end{defn}

\noindent The quotient set $^{*}S/\sim$ and $\beta S$ are homeomorphic.
As is a closed set, it contains idempotents. For each idempotent $\alpha\in{}^{*}S/\sim$,
$\alpha+\,^{*}\!\alpha\sim\alpha$. Now each idempotent in $\beta S$
is correspond to an idempotent $\alpha\in{}^{*}S/\sim$ and so if
$A$ is an $IP^{*}$ set, $\alpha\in\,^{*}\!A$ for all $u$-idempotent
$\alpha$ in $^{*}S/\sim$.

\section{Our results}
\begin{lem}
\label{lem}Let $A$ be an $IP^{*}$set in $\left(\mathbb{N},+\right)$
then for any $n\in\mathbb{N}$, $n^{-1}A$ is also $IP^{*}$ set in
$\left(\mathbb{N},+\right)$.
\end{lem}

\begin{proof}
Let $\alpha\in^{*}\mathbb{N}$ be any $u$-idempoptent. As $A$ is
an $IP^{*}$ set, $A\in\mathcal{U}_{\alpha}.$ As $\alpha$ is an
idempotent, $n\alpha$ is also an $u$-idempoptent for all $n\in\mathbb{N}$.
As $A$ is an $IP^{*}$set, $A\in\mathcal{U}_{n\alpha}$ for all $n\in\mathbb{N}$.
But then $n^{-1}A\in\mathcal{U}_{\alpha}$ \cite[Exercise 3.2, page 45]{key-1-1}.
As this holds for all $u$-idempoptent $\alpha$, we have $n^{-1}A$
is an $IP^{*}$ set.
\end{proof}

\begin{fact}
\label{fact} In the above proof, we have seen that if $A$ is an
$IP^{*}$ set, then $n\alpha\in\,^{*}\!A$ for all $n\in\mathbb{N}$,
and $u$-idempotent $\alpha$.
\end{fact}
\noindent In the following proof let us denote the set of all finite subsets of $\mathbb{N}$ by $\mathcal{P}_{f}\left(\mathbb{N}\right)$.
\begin{proof}[\textbf{\textit{Proof of Theorem \ref{Main}:}}]
 Let, $B=FS\langle x_{n}\rangle_{n=1}^{\infty}$ and choose an $u$-idempotent
$\alpha\in\,^{*}\!B.$ As $A$ is an $IP^{*}$ set, $\alpha\in\,^{*}\!A$.
Now $\alpha\sim\alpha+\,^{*}\!\alpha$, $\alpha+\,^{*}\!\alpha\in\,^{**}\!A$.
So, by transfer choose $y_{1}=\sum_{t\in H_{1}}x_{t}$ such that,
\[
y_{1}\in A
\]
\[
y_{1}+\alpha\in\,^{*}\!A.
\]

From the above equation, and fact \ref{fact} we have,
\[
y_{1}\in A
\]
\[
\alpha\in\,^{*}\!A
\]
\[
y_{1}+\alpha\in\,^{*}\!A
\]
\[
y_{1}+\,^{*}\!\alpha\in\,^{**}\!A\left(\text{by transfer principle}\right)
\]
\[
\alpha+\,^{*}\!\alpha\in\,^{**}\!A
\]
\[
y_{1}+\alpha+\,^{*}\!\alpha\in\,^{**}\!A
\]
\[
y_{1}\alpha\in\,^{*}\!A.
\]
Again by transfer principle, choose $H_{2}\in\mathcal{P}_{f}\left(\mathbb{N}\right)$
where $\text{max}H_{1}<\text{min}H_{2}$, $y_{2}=\sum_{t\in H_{2}}x_{t}$
such that

\[
y_{1}\in A
\]
\[
y_{2}\in A
\]
\[
y_{1}+y_{2}\in A
\]
\[
y_{1}+\alpha\in\,^{*}\!A
\]
\[
y_{2}+\alpha\in\,^{*}\!A
\]
\[
y_{1}+y_{2}+\alpha\in\,^{*}\!A
\]
\[
y_{1}y_{2}\in A.
\]
Now by induction, assume that for $m\in\mathbb{N}$, we have already
chosen $\left(H_{i}\right)_{i=1}^{m}$ such that $\text{max}H_{i}<\text{min}H_{j}$
if $i,j\in\left\{ 1,2,\ldots,m\right\} $ and $i<j$ and $FS\langle y_{i}\rangle_{i=1}^{m}\cup FP\langle y_{i}\rangle_{i=1}^{m}\subset A$
satisfying the following conditions:
\[
t+\alpha\in\,^{*}\!A
\]
\[
s\alpha\in\,^{*}\!A
\]
for all $t\in FS\langle y_{i}\rangle_{i=1}^{m}$ and $s\in FP\langle y_{i}\rangle_{i=1}^{m}$.
Now from the above equation we have,
\[
\alpha\in\,^{*}\!A
\]
\[
t+\alpha\in\,^{*}\!A
\]
\[
t+\,^{*}\!\alpha\in\,^{**}\!A\,\left(\text{by transfer principle}\right)
\]
\[
\alpha+\,^{*}\!\alpha\in\,^{**}\!A
\]
\[
t+\alpha+\,^{*}\!\alpha\in\,^{**}\!A
\]
\[
s\alpha\in\,^{*}\!A.
\]
Again by transfer principle, choose $H_{m+1}\in\mathcal{P}_{f}\left(\mathbb{N}\right)$
where $\text{min}H_{m+1}>\text{max}H_{m}$, $y_{m+1}=\sum_{t\in H_{m+1}}x_{t}$
such that
\[
y_{m+1}\in A
\]
\[
t+y_{m+1}\in A
\]
\[
t+\alpha\in\,^{*}\!A
\]
\[
y_{m+1}+\alpha\in\,^{*}\!A
\]
\[
t+y_{m+1}+\alpha\in\,^{*}\!A
\]
\[
sy_{m+1}\in A.
\]
Hence $FS\langle y_{i}\rangle_{i=1}^{m+1}\cup FP\langle y_{i}\rangle_{i=1}^{m+1}\subset A$
and for all $t^{\prime}\in FS\langle y_{i}\rangle_{i=1}^{m}$ and
$s^{\prime}\in FP\langle y_{i}\rangle_{i=1}^{m}$,
\[
t^{\prime}+\alpha\in\,^{*}\!A
\]
\[
s^{\prime}\alpha\in\,^{*}\!A\,\left(\text{from fact \ref{fact}}\right).
\]
So the induction is complete and the theorem is proved.\\
\end{proof}

\noindent \textbf{Acknowledgment: }The author acknowledges the grant UGC-NET
SRF fellowship with id no. 421333 of CSIR-UGC NET December 2016.

\end{document}